\documentclass[reqno,11pt]{amsart}
\usepackage{amsmath,amssymb,latexsym,soul,cite,mathrsfs}
\usepackage{color,enumitem,graphicx}
\usepackage[colorlinks=true,urlcolor=blue,
citecolor=red,linkcolor=blue,linktocpage,pdfpagelabels,
bookmarksnumbered,bookmarksopen]{hyperref}
\usepackage[english]{babel}
\usepackage[left=2.7cm,right=2.7cm,top=2.9cm,bottom=2.9cm]{geometry}
\usepackage[hyperpageref]{backref}
\usepackage{lipsum}
\usepackage{enumerate}
\usepackage{tcolorbox}

\usepackage{tikz}
\usetikzlibrary{arrows.meta}
\usetikzlibrary{shapes.geometric, arrows}

\usepackage{tabularx}

 \usepackage{mathtools}

\newtheorem{theorem}{Theorem}[section]
\newtheorem{lemma}[theorem]{Lemma}

\newtheorem{corollary}[theorem]{Corollary}

\newtheorem{proposition}[theorem]{Proposition}
\newtheorem{remark}[theorem]{Remark}
\newtheorem{definition}[theorem]{Definition}

\DeclarePairedDelimiter{\abs}{\lvert \hspace{.7pt}}{\rvert}

\newcommand{\innerthmname}{}

\theoremstyle{definition}

\makeatletter
\def\namedlabel#1#2{\begingroup
	#2%
	\def\@currentlabel{#2}%
	\phantomsection\label{#1}\endgroup
}
\makeatother

\newcommand{\Sp}{\mathbb{S}}
\newcommand{\ep}{\varepsilon}

\newcommand{\s}{\hspace{7pt}}

\newcommand{\R}{\mathbb{R}}
\newcommand{\C}{\mathbb{C}}

\newcommand{\vsp}{\vspace{2pt}}
\newcommand{\ns}{\hspace{-3pt}}

\newcommand{\Z}{\mathcal{Z}}
\newcommand{\N}{\mathcal{N}}

\newcommand{\mres}{\mathbin{\vrule height 1.6ex depth 0pt width
0.13ex\vrule height 0.13ex depth 0pt width 1.3ex}}

\DeclareMathOperator{\supp}{supp}
\DeclareMathOperator{\dist}{dist}

\title[Saddle solution to the complex Ginzburg-Landau system in three dimensions]{Saddle solution to the complex Ginzburg-Landau system in three dimensions}

\author[M. Caselli]{Michele Caselli}
\author[N. Picenni]{Nicola Picenni}

\address[M. Caselli]{
	Scuola Normale Superiore di Pisa
	\newline\indent 
	Piazza dei Cavalieri 7, 56126 Pisa, Italy}
\email{\href{mailto:michele.caselli@sns.it}{michele.caselli@sns.it}}

\address[N. Picenni]{
	University of Pisa, Department of Mathematics
	\newline\indent 
	Largo Bruno Pontecorvo 5, 56127 Pisa, Italy}
\email{\href{mailto:nicola.picenni@unipi.it}{nicola.picenni@unipi.it}}

\makeatletter
\def\l@subsection{\@tocline{2}{0pt}{2.5pc}{5pc}{}}
\def\l@subsubsection{\@tocline{2}{0pt}{5pc}{7.5pc}{}}
\makeatother

\begin{document}
	
	\begin{abstract}
		We construct an entire saddle solution to the non-magnetic complex Ginzburg-Landau system in three dimensions, whose zero set is the union of two perpendicular, intersecting lines, and whose blow-downs concentrate on these lines with multiplicity one.
	\end{abstract}
	
	\maketitle

  \tableofcontents



\vsp


\section{Introduction}

In this work, we prove the existence of a smooth, saddle-type solution $u: \R^3 \to \C$ to the complex Ginzburg-Landau system
\begin{equation}\tag{\(GL \)}\label{GL}
    -\Delta u = u(1-|u|^2),
\end{equation} 
whose zero set $\Z(u) :=\{u=0\}$ (also called ``vorticity set") is the union of the two perpendicular, intersecting lines
\begin{equation*}
    X := \{xy=0 \} \cap \{ z=0 \} .
\end{equation*}
In addition, this solution enjoys particular symmetries (see Definition \ref{sobdef}) and its blow-downs concentrate on $X$ with multiplicity one, in the sense that \eqref{eq: main energy bound for u} below holds. 

\vsp 
Solutions to the system \eqref{GL} in a domain $\Omega \subset \R^n$ arise as critical points of the complex Ginzburg-Landau energy
\begin{equation}\label{GL energy}
            E(u, \Omega) := \int_{\Omega}  \frac{1}{2}|\nabla u|^2 + \frac{(1-|u|^2)^2}{4}   , \quad \mbox{for } u \in H^1(\Omega; \C) . 
  \end{equation}

With this notation, the main result of our work is the following.
\begin{theorem}\label{mainteo}
    There exists a smooth solution $u:\R^3 \to \C$ to \eqref{GL} with $|u|<1$ and zero set $\Z(u) = X$. Moreover, $u$ obeys the symmetries in Definition \ref{sobdef} and 
        \begin{equation}\label{eq: main energy bound for u}
            \lim_{r \to \infty} \frac{E(u, B_r)}{r\log(r)} = 4\pi .
        \end{equation}
\end{theorem}

In particular, in the varifold viewpoint for the normalized energy densities 
\begin{equation}\label{eq: norm en dens}
            \mu_r := \frac{e_{1/r}(u_r)}{\pi \log(r) }dx , \quad \mbox{where } e_{1/r}(u_r) := \frac{1}{2}|\nabla u_r|^2 + r^2\frac{(1-|u_r |^2)^2}{4} , 
        \end{equation}
of the blow-downs $u_r(x) := u(r x)$, we obtain the following corollary. 

\begin{corollary}\label{cor: blow down}
    The normalized energy densities \eqref{eq: norm en dens} converge to a limit supported on $X$ with unit density. That is 
\begin{equation*}
           \mu_r \rightharpoonup \mu_\infty := \mathcal{H}^1 \mres (X \cap B_1) ,  
        \end{equation*}
in $C_c^0(B_1)^*$ as $r \to +\infty$. 
\end{corollary}

Our result provides the first construction of an entire saddle solution to the complex Ginzburg-Landau system in three dimensions with a singular zero set and can be thought of as a codimension two analogue of the classical work \cite{DangFife92}, where the authors build an entire saddle solution $u: \R^2 \to (-1,1)$ of the Allen-Cahn equation $- \Delta u = u(1-u^2)$ with zero set $\{xy=0\}$.

\vsp 

We call our $u$ a ``saddle solution" since, as in the case of codimension one in \cite{DangFife92}, we expect it to be unstable. Indeed, the solution of \cite{DangFife92} is unstable and it has index $1$ (c.f. \cite{Schatzman}).

\subsection{A brief overview of the Ginzburg-Landau system}

The system (\ref{GL}) and the energy functional (\ref{GL energy}) arise as a simplified version (with a vanishing magnetic field) of the Ginzburg-Landau model for superconductivity, and have attracted considerable attention in the mathematical community following the seminal work \cite{BBHbook}. A particular effort has been devoted to the asymptotic analysis of the $\ep$-rescaled versions
$$- \Delta u = \frac{1}{\ep^2}u(1-|u|^2),\qquad E_\ep(u,\Omega):= \int_{\Omega}  \frac{1}{2}|\nabla u|^2 + \frac{(1-|u|^2)^2}{4\ep^2}  $$
both as a relaxation for the problem of finding singular $\Sp^1$-valued harmonic maps when there are no finite energy competitors, and as a codimension two analogue of the Allen-Cahn approximation of the area functional to find minimal surfaces.

\vsp 

After nearly two decades of research, and many highly influential papers \cite{LinRiv1,LinRiv2,BBO-asymptotics,ABO,BOSannals,DanielExistence}, a rather complete comprehension of this problem (and of the pathologies it may display) came with the work of Pigati-Stern \cite{Quantnonquant}, whose main result is (partially) quoted in Theorem~\ref{thm: general convergence theorem}.

\vsp 

A different but of course strongly related question concerns the characterization of entire solutions to (\ref{GL}), under suitable assumptions on the behavior of $u$ at infinity, either in terms of energy growth or of decay of $1-|u|$.

\vsp 
In the two dimensional case, it is conjectured that, up to isometries, the only solutions such that $|u|\to 1$ at infinity are the planar degree $d\in\mathbb{Z}$ vortices, namely a sequence of solutions that can be written in polar coordinate in the form $v_d(r,\theta)=\rho_d(r)e^{id\theta}$, where $\rho_d:[0,+\infty)\to [0,1]$ satisfies a suitable ODE (see \cite{radialsol1,radialsol2}). This long-standing conjecture is still open, even though partial results under additional assumptions are known (see \cite[Section~2]{Brezis-op}).

\vsp

In three dimensions, very few nontrivial (that is, neither constant nor the straight extension of the planar vortices $v_d$) solutions of the complex Ginzburg-Landau system are known. To our knowledge, beyond these trivial examples, the only rigorously constructed entire nontrivial solutions to date are the interacting helical vortex filaments families built by D\'avila-del Pino-Medina-Rodiac in \cite{helicalvortfilaments}, featuring a vorticity set of $n\ge2$ interacting helices. These solutions also furnished a negative answer to the Ginzburg–Landau analogue of the Gibbons conjecture: whether an entire solution $u$ of \eqref{GL} such that 
\begin{equation*}
    \lim_{|z| \to \infty } |u(z,t)| = 1 , \s \mbox{uniformly in } t \in \R
\end{equation*}
is necessarily a function of $z$ only. 

\vsp

In addition, regarding solutions in a bounded domain or a compact Riemannian manifold, two more families of solutions are known: the nearly parallel vortex filament solutions in bounded cylindrical domains of \cite{npvortfilaments} and the solutions built in \cite{ColJerSetrnGeodesicGL} concentrating on a nondegenerate geodesic in a $3$-dimensional closed manifold. 

\vsp

Concerning Corollary~\ref{cor: blow down}, we point out that the integrality and unit-multiplicity of the blow-down varifold in Corollary \ref{cor: blow down} are nontrivial. Indeed, in dimensions $n\ge 3$, Pigati–Stern \cite{Quantnonquant} proved that for general families of Ginzburg–Landau critical points within the natural $O(\abs{\log(\ep)})$ energy regime, the energy–concentration varifold arising in the limit need not be integral: the limit density can take values in $\{1\} \cup [2, +\infty)$ and they construct examples realizing every prescribed density $\theta \ge 2$. Their examples are obtained suitably blowing down the interacting helical vortex filament solutions of \cite{helicalvortfilaments}. In contrast, the solution given by our Theorem \ref{mainteo} yields a stationary integral varifold of unit density on each line. 

\vsp 

We also remark that our solution is not an unconstrained local minimizer (i.e., it is not a minimizer of \eqref{GL energy} in compact sets with its own boundary datum), so it does not contradict the conjecture (see \cite[Open Problem 5]{Pisante14} or \cite{SandierShafrir}) that the only solutions of \eqref{GL} which are (unconstrained) local minimizers and satisfy
  \begin{equation*}
            \limsup_{ r \to \infty } \frac{E(u, B_r)}{ r \log(r)}   <+\infty   ,
  \end{equation*}
  are $2$-dimensional vortices in some direction. That is, up to an isometry on $\R^3$ and an orthogonal transformation in $\C$ they can be written as $u(x,y,z)=v_d(x,y)$. By the symmetries of the zero set $\mathcal{Z}(u)=X$ of our solution, it is clear that it cannot be written in such a form. Nevertheless, by construction, $u$ is a local minimizer in a suitable space with symmetries $\widetilde H^1(\R^3; \C)$ as defined in Definition \ref{sobdef}. 

\vsp

Finally, let us mention several possible questions which naturally arise in view of our result: determining the Morse index and stability of our saddle solution, extending the construction to other symmetric line-networks or periodic arrays, and clarifying uniqueness within the symmetry class, together with a classification of entire solutions at the $4\pi$ energy growth regime (cf. \cite{SandierShafrir}).

\vsp 
Moreover, we remark that our solution can trivially be extended to higher dimensions, providing examples of entire solutions in $\R^n$ whose zero set consists of two orthogonal $(n-2)$-planes intersecting along an $(n-3)$-plane. We expect that it should be possible to adapt our method to find further classes of solutions in higher dimension (or codimension) whose zero set consists of different arrangements of subspaces (for example, two $2$-planes intersecting in a point in $\R^4$).

\subsection{Structure of the proof}
Let us briefly describe the strategy we use to prove the existence of $u$ in Theorem \ref{mainteo}. The first part of the proof is quite standard, while the second part, in which we prove \eqref{eq: main energy bound for u}, requires a nontrivial idea due to the lack of a monotonicity formula at the correct rate (see below).

    \vsp 
    First, in \hyperref[sec:preliminaries]{Section~\ref{sec:preliminaries}} we introduce the precise setting of symmetric functions, then in \hyperref[sec: minimization in B_R]{Section~\ref*{sec: minimization in B_R}}, working in a suitable Sobolev space with symmetries $\widetilde H^1(B_R; \C)$ (defined in Definition \ref{sobdef}) that force every competitor to vanish on $X$, we minimize the Ginzburg-Landau functional \eqref{GL energy} in $B_R:=B_R(0) \subset \R^3$ with a generic Dirichlet boundary condition and we prove that a minimizer $u_R$ exists, solves \eqref{GL} in the entire ball $B_R$, satisfies the symmetries in Definition \ref{sobdef}, and vanishes exactly on $X\cap B_R$. 

    Then, in \hyperref[sec: g constr]{Section~\ref*{sec: g constr}}, we construct a suitable family of boundary data $g_R:\partial B_R\to\C$ for which the minimizers $u_R$ satisfy
 \begin{equation}\label{eq: intro bound u_R}
            \limsup_{R \to \infty} \frac{E(u_R, B_R)}{R\log(R)} \leq 4\pi,
        \end{equation}        
and in \hyperref[sec: global sol]{Section~\ref*{sec: global sol}}, we show that a subsequential limit of $u_R$ as $R\to +\infty$ converges to a limit function $u:\R^3 \to \C$ that solves \eqref{GL}, and we compute its degree around $X$.

\vsp 
Lastly, in \hyperref[sec: sharp energy asymptotics]{Section~\ref*{sec: sharp energy asymptotics}}, which is the most delicate part, we prove the sharp energy bound \eqref{eq: main energy bound for u} for our solution $u$, by a careful analysis of how the energy of $u_R$ propagate at scales smaller than $R$. This crucial step presents some specific difficulties for the (non-magnetic) Ginzburg-Landau energy \eqref{GL energy}. Indeed, in our setting the usual monotonicity formula (see, for example, \cite[Lemma~II.2]{BBO-asymptotics}) states that 
\begin{equation*}
    \frac{E(u_R, B_r)}{r} \le \frac{E(u_R, B_R)}{R} , \quad \forall \, r \le R , 
\end{equation*}
so it propagates energy from a large ball down to smaller scales with the wrong rate, and it cannot recover the optimal $ r\log(r)$ growth. More precisely, from (\ref{eq: intro bound u_R}) one only gets $ E(u_R, B_r) \lesssim  r\log (R) $; letting \(R\to\infty\) makes the right-hand side blow up, so this yields no growth information on the energy of $u$ in $B_r$. This makes our problem totally different from the codimension-one Ginzburg-Landau or the Yang-Mills-Higgs cases, where the monotonicity formula is compatible with the natural scaling and directly transmits sharp bounds across smaller radii (see \cite{Ilmanen-AC2Brakke,HutTon} for monotonicity formulas in the scalar case and \cite{PigStern, DPHP} for the Yang-Mills-Higgs case).

\vsp 
This mismatch forces an ad hoc propagation scheme for minimizers: in an annulus $B_{(1+\delta)r} \setminus B_{r}$ we patch $u_R$ with competitors that encode the spherical boundary phase of $g_R$ and a carefully chosen modulus, leading to a decay estimate that is (essentially)
\begin{equation*}
    E(u,B_r)\le 4\pi r\log(r)  +  C\sqrt{r\,E(u,B_{(1+\delta)r})}.
\end{equation*}
From here, an elementary iteration lemma (Lemma \ref{lem: growth reabsorbing lemma}) then converts this into the global upper bound $E(u, B_r) \le 4\pi r\log(r)+ o(r\log(r))$. 

\vsp 

Finally, once the sharp upper bound is established, the corresponding lower bound and the characterization of the blow-down follow quite easily from some deep results from \cite{Quantnonquant}, which we recall in the appendix.


\section{Preliminaries and notation} \label{sec:preliminaries}

We begin by introducing some notation. For $n\ge 3$, we denote by $B_R(x)$ a ball in $\R^n$ with center $x\in \R^n$ and radius $R>0$. When $n=2$, we denote by $D_R(x)$ the two-dimensional disk of radius $R>0$ and center $x \in \R^2$. In the case of $x=0$, we denote for short 
\begin{equation*}
    B_R = B_R(0) , \s   D_R = D_R(0) .
\end{equation*}

We describe now the setting of the function spaces with symmetries that we will work with. Let the coordinate reflections be
\begin{equation*}
    \mathcal R_x(x,y,z)=(-x,y,z),\qquad
\mathcal R_y(x,y,z)=(x,-y,z),\qquad
\mathcal R_z(x,y,z)=(x,y,-z),
\end{equation*}
and let $H^1(B_R; \C)$ be the standard Sobolev space $W^{1,2}(B_R; \C)$ for functions $u:B_R \subset \R^3 \to \C$.

\begin{definition}\label{sobdef} 
     We denote by $ \widetilde H^1(B_R; \C)$ the set of functions $ u\in H^1(B_R,; \C) $ that satisfy the following symmetries almost everywhere in $B_R$: 
\begin{equation*}
u\circ \mathcal R_x=- \overline u,\qquad u\circ \mathcal R_y=- \overline u,\qquad u\circ \mathcal R_z= \overline u.
\end{equation*}
\end{definition}

\begin{remark}\label{rem:symmetries}
We point out that any continuous function $u=u_1+iu_2 \in \widetilde{H}^1(B_R; \C)$ necessarily satisfies
\begin{equation}\label{eq: where u zero}
    u_1=0 \mbox{ on } \{x=0\} \cup \{y=0\}, \quad u_2=0 \mbox{ on } \{z=0\} ,  
\end{equation}
which implies in particular that $X\cap B_R\subset \Z(u)$. Actually, the same is true also for the precise representative of any $u\in \widetilde{H}^1(B_R; \C)$, but we do not need this in the sequel.

Moreover, if $u\in C^1(B_R;\C) \cap \widetilde{H}^1(B_R; \C)$, then it necessarily satisfies also
 \begin{equation}\label{rand1}
\begin{aligned}
    & \partial_x u_2 =0 , \s && \textnormal{on} \s \{x=0\} , \\ & \partial_y u_2 =0 ,  \s &&  \textnormal{on} \s \{y=0\} , \\ & \partial_z u_1 =0 ,  \s &&  \textnormal{on} \s \{z=0\} .
\end{aligned}
\end{equation}
    
\end{remark}

Let us set
\begin{equation*}
     Q := \{ x>0\,, y>0\,, z>0\} \,,  \mbox{ and } Q_R:=Q\cap B_R ,
\end{equation*}
and let also 
\begin{align*}
    \partial^s Q_R & := \partial B_R \cap \partial Q_R , \\
    \partial^x Q_R & := \{x=0\} \cap \partial Q_R , \\
    \partial^y Q_R & := \{y=0\} \cap \partial Q_R ,  \\
    \partial^z Q_R & := \{z=0\} \cap \partial Q_R .
\end{align*}
Observe that the value of any $u\in \widetilde H^1(B_R; \C)$ is completely determined by its values on the octant $Q_R$ just by reflecting these values according to the symmetries in Definition \ref{sobdef}. Vice versa, every map $v\in H^1(Q_R; \C)$ that satisfies \eqref{eq: where u zero} can be extended to a function in $\widetilde H^1(B_R; \C)$ by reflecting with our symmetries. 

Moreover, for every $ u \in  \widetilde H^1(B_R; \C)$ we have that 
\begin{equation*}
    E(u, B_R) = 8 E(u, Q_R) .  
\end{equation*}


\section{Minimization in \texorpdfstring{$B_R$}{} for general boundary datum}\label{sec: minimization in B_R}

This section is entirely dedicated to the proof of the following result.
\begin{proposition}\label{existence} 
    Let $g = g_1+ig_2 \in \widetilde H^{1/2}(\partial B_R ; \C)$ be such that $|g|\le 1$, $g_1 \ge 0 $ and $ g_2 \ge 0$ on $\partial^s Q_R$. Then, the minimum 
\begin{equation}\label{minprob}
     \min \Big\{ E(v, B_R ) \, : \,  v \in \widetilde H^1_g(B_R; \C) \Big\} 
\end{equation}
is achieved by a smooth map $u \in C^\infty(B_R; \C)$ that solves 
\begin{equation}\label{GL equation}
    \begin{cases} -\Delta u = u (1-|u|^2)   &   \mbox{in} \s B_R  ,  \\  \s u=g &  \mbox{on} \s \partial B_R  . \end{cases}
\end{equation}
Moreover $u=u_1+iu_2$ satisfies $|u|<1$, $ u_1, u_2>0$ in $ Q_R$ and $\Z(u)=X \cap B_R$. 
\end{proposition}

\begin{remark}\label{modleone}
     It is a simple consequence of the strong maximum principle that every solution of \eqref{GL equation}, not only the minimizer,  with $|g|\le 1$ on $\partial B_R$ satisfies $|u| \le 1$. Indeed, the function $v:=|u|^2-1$ solves
 \begin{equation*}
     -\Delta v + 2|u|^2 v \le 0 , \s \mbox{in } B_R.
 \end{equation*}
Since $v = |g|^2-1 \le 0 $ on $\partial B_R$ and $2|u|^2 \ge 0$, by the strong maximum principle $v=|u|^2-1 < 0$ in $B_R$, unless $|u|\equiv 1$, which is impossible because of (\ref{eq: where u zero}).
\end{remark}

\begin{proof}[Proof of Proposition \ref{existence}] 
The proof will be divided into four steps. 
\begin{itemize}[label={}, leftmargin=15pt]
\setlength\itemsep{3pt}
    \item\textbf{Step 1.} There exists a minimizer $u=u_1+iu_2 \in \widetilde H^1_g(B_R; \C)$ of the minimum problem \eqref{minprob}, and satisfies $|u|\le 1$ and $u_1, u_2 \ge 0$ in $Q_R$. 
    \item\textbf{Step 2.} Any minimizer $u$ is smooth in the octant $Q_R$ and solves $-\Delta u = u(1-|u|^2)$ in $Q_R$. 
    \item\textbf{Step 3.} Any minimizer $u$ is smooth and is a solution of $-\Delta u = u(1-|u|^2)$ in all of $B_R$. That is, it solves $-\Delta u = u(1-|u|^2)$ also across the faces of $Q_R$, across the axes, and at the origin.
    \item\textbf{Step 4.} The zero set of $u$ is exactly $X \cap B_R$. 
\end{itemize}
\begin{proof}[Proof of Step 1.] Since $E(\,\cdot \,, B_R) \ge 0$ is nonnegative, let $\{v_k \} \subset \widetilde H^1_g(B_R; \C)$ be a minimizing sequence. First, we claim that, up to replacing $v_k$ with another minimizing sequnce, we can assume $|v_k|\le 1$ and $(v_k)_1, (v_k)_2 \ge 0$ in $ Q_R$.

Indeed, replacing $v_k$ with $ w_k := |(v_k)_1|+i |(v_k)_2| $ in $Q_R$ and reflecting with the symmetries in Definition \ref{sobdef}, we get that $\{w_k\} \subset \widetilde H^1_g(B_R; \C)$ is still a minimizing sequence since $E(v_k, B_R)=E(w_k, B_R)$. Furthermore, replacing $w_k$ with $\widetilde w_k := w_k \min\{1, 1/|w_k|\} $ achieves $|\widetilde w_k| \le 1$ in $B_R$ and is still a minimizing sequence. Indeed, for the potential part of the Ginzburg-Landau energy 
\begin{equation}\label{rand3}
    \int_{B_R} \frac{(1-|\widetilde w_k|^2)^2}{4}  \le  \int_{B_R} \frac{(1-|w_k|^2)^2}{4} ,
\end{equation}
since $|\widetilde w_k| = \min\{ |w_k|, 1 \}$ by definition. For the Dirichlet part
\begin{equation}\label{rand4}
    \int_{B_R} |\nabla \widetilde{w}_k|^2 = \int_{B_R} |\nabla |\widetilde{w}_k||^2 + |\widetilde{w}_k|^2 \left|\nabla \frac{\widetilde{w}_k}{|\widetilde{w}_k|}\right|^2 \le \int_{B_R} |\nabla |w_k||^2 + |w_k|^2 \left|\nabla \frac{w_k}{|w_k|}\right|^2 = \int_{B_R} |\nabla w_k|^2.
\end{equation}

Putting together \eqref{rand4} and \eqref{rand3} we obtain that
\begin{equation*}
    E(\widetilde w_k, B_R) \le E(w_k, B_R) .
\end{equation*}
Hence $\{\widetilde w_k\}$ is also a minimizing sequence, and possibly replacing $v_k$ with $\widetilde w_k$ we assume that $v_k$ has $|v_k|\le 1$ and $(v_k)_1, (v_k)_2 \ge 0$ in $Q_R$.

Since $E(v_k, B_R) \le C  $ and $|v_k| \le 1$, we have that $\| v_k \|_{H^1(B_R; \C)} \le C(R)$ is uniformly bounded in $k$. Hence, there exists a subsequence (that we do not relabel) and $u \in H^1(B_R; \C)$ such that $v_k \rightharpoonup u$ weakly in $H^1(B_R)$, $v_k \to u $ in $L^2(B_R)$ and almost everywhere. Since $\widetilde H^1_g(B_R; \C)$ is a closed subspace of $H^1(B_R)$, we also have $u \in \widetilde H^1_g(B_R; \C)$. By the lower-semicontinuity of the seminorm under weak convergence and Fatou's lemma
\begin{equation*}
    E(u, B_R) \le \liminf_{k\to \infty} E(v_k, B_R) = \inf_{ v \in \widetilde H^1_g(B_R; \C)} E(v, B_R ) ,
\end{equation*}
thus $u$ is a minimizer in $\widetilde H^1_g(B_R; \C)$. Moreover, by the a.e. convergence of $v_k$ to $u$ we also get $|u|\le 1$ and $u_1, u_2\ge 0$ in $Q_R$. 
\end{proof}

\begin{proof}[Proof of Step 2.] Choose any $\varphi \in C^\infty_c( Q_R,\C)$. Reflecting $\varphi$ with the symmetries of Definition \ref{sobdef} produces a test function $ \Phi \in C^\infty_c(B_R,\C) \cap \widetilde H^1_0(B_R; \C )$. Hence for every $\ep \in (-1, 1) $ we have $u+\ep \Phi  \in \widetilde H^1_g(B_R; \C ) $, and by the minimality of $u$ in $\widetilde H^1_g(B_R; \C )$ we immediately get 
\begin{align*}
    0=\frac{d}{d\ep}\bigg|_{\ep=0} E(u +\ep \Phi , B_R) & =  \int_{B_R} \nabla u \cdot \nabla { \Phi }  +  \int_{B_R} (|u|^2-1) u \cdot {\Phi}  \\ & = 8 \left(  \int_{Q_R} \nabla u \cdot \nabla \varphi  +   \int_{Q_R} (|u|^2-1) u \cdot \varphi  \right) .
\end{align*}
    
    Hence, $u$ is a weak solution of $-\Delta u =u(1-|u|^2)$ in $Q_R$, and by standard elliptic regularity, $u$ is smooth and is a strong solution in $Q_R$.  
\end{proof}

\begin{proof}[Proof of Step 3.] 

Our goal is to show that $u$ satisfies the weak form of the full Ginzburg-Landau 
system on the whole ball $B_R$ with no symmetry restriction on the test functions. This follows from the Palais's Principle of Symmetric Criticality \cite{PalaisSymcrit} in the case of a $0$-dimensional finite Lie group. Our setting is so concrete and specific that we prefer giving a short proof of this fact explicitly instead of relying on this abstract principle. Once it is proved that $u$ is an unconstrained weak solution in the interior of $B_R$, standard interior elliptic regularity yields smoothness across the planes.

Fix $\varphi\in C_c^\infty(B_R;\mathbb C)$. We stress that we do not impose any symmetry on $\varphi$. Exploiting the symmetries of $u$ and many changes of variables, we find that
\begin{equation}\label{test_phi_psi}
\int_{B_R} \nabla u \cdot \nabla \varphi - (1-|u|^2) u \cdot \varphi = \int_{Q_R} \nabla u \cdot \nabla \Phi - (1-|u|^2)u\cdot \Phi,
\end{equation}
where
$$\Phi : =\varphi - \overline{\varphi}\circ \mathcal{R}_x -  \overline{\varphi}\circ \mathcal{R}_y + \overline{\varphi}\circ \mathcal{R}_z + \varphi\circ\mathcal{R}_x\circ\mathcal{R}_y-\varphi\circ\mathcal{R}_x\circ\mathcal{R}_z -\varphi\circ\mathcal{R}_y\circ\mathcal{R}_z +
\overline{\varphi}\circ\mathcal{R}_x\circ\mathcal{R}_y\circ\mathcal{R}_z.$$
Now we observe that $\Phi$ satisfies the symmetries in Definition~\ref{sobdef}, namely $\Phi \in \widetilde{H}^1 (B_R;\C)\cap C^\infty _c(B_R;\C)$, and hence it is an admissible variation to test the criticality of $u$. As a consequence, from (\ref{test_phi_psi}) we deduce that
$$\int_{B_R} \nabla u \cdot \nabla \varphi - (1-|u|^2)u\cdot \varphi=\frac{1}{8} \int_{B_R} \nabla u \cdot \nabla \Phi - (1-|u|^2)u\cdot \Phi = 0,$$
namely, $u$ is a weak solution of $\eqref{GL}$ in all of $B_R$, and by standard regularity theory, it is also a strong solution in all of $B_R$. This finishes the proof of Step 3.

\end{proof}

\begin{proof}[Proof of Step 4] 
Recall that every function $v\in \widetilde H^1_g(B_R; \C)$ vanishes on $X \cap B_R$. Hence, to prove Step 4, it is enough to show that $u$ does not vanish outside this set.  

Taking the real part in \eqref{GL equation} gives
    \begin{equation*}
        \mathcal{L} u_1 := \Delta u_1 + c(x)u_1 = 0  , \s \textnormal{in} \,\, Q_R ,
    \end{equation*}
   where $c(x)=1-|u|^2 \ge 0$ since $|u| \le 1$ in $Q_R$. Then, by the strong maximum princple $u_1>0$ in $Q_R$ unless $u_1 \equiv 0$ in $Q_R$, but since $u_1 = g_1 $ on $\partial^s Q_R$ and $g_1$ does not vanish identically on $\partial^s Q_R$, we get that $u_1>0$. Moreover, since $u_2 \ge 0$ in $Q_R$, a completely similar argument with the maximum principle implies that also $u_2>0$ in $Q_R$. 

    Now, assume by contradiction that there exists $x_\circ \in B_R \setminus X $ such that $u(x_\circ)=0$. We distinguish two cases. 

 \begin{itemize}[label={}, leftmargin=15pt]
\setlength\itemsep{3pt}

     \item\textbf{Case 1.} $x_\circ \in B_R \setminus  \{z= 0\} $. By the symmetries, we can assume $x_\circ \in B_R \cap  \{z> 0\} $. Since $u_2 \ge 0$ on $B_R \cap  \{z> 0\}$ and $\mathcal{L} u_2 = 0 $ on this set, by the strong maximum principle we would get $u_2 \equiv 0$ in $B_R \cap  \{z> 0\}$, which is impossible for our boundary datum $g$.

    \item\textbf{Case 2.} $x_\circ \in B_R \cap \big( \{z=0\} )\setminus X  \big)$. By the symmetries, we can assume that $ x_\circ \in \partial^z Q_R \setminus X$. Recall that $u_1 >0$ in $Q_R$. By the symmetries in the definition of $\widetilde H^1_g(B_R; \C)$ we have that $u_1$ is even across $\partial^z Q_R$ and $\partial_z u_1=0$ on $\partial^z Q_R$ by \eqref{rand1}. Since $0=u_1(x_\circ) < u_1(x)$ for every $x\in Q_R$, this contradicts Hopf's boundary point lemma (e.g., Lemma 3.4 in \cite{GilTru} applied to $-u_1$ and $\Omega=Q_R$).
    
\end{itemize}
This concludes the proof of Step 4. 
\end{proof}

 By Steps 1-4 above, we conclude the proof of Proposition \ref{existence}.

\end{proof}

\begin{remark} We could have obtained $u$ with the same properties as above with a different, unconstrained minimization problem with mixed Dirichlet-Neumann boundary conditions. Indeed, for $u\in H^1(Q_R; \C )$ consider the following set of boundary conditions on $\partial Q_R$:
\begin{equation}\label{dir}
    \begin{cases} u_1=0   , \,\,  \partial_\nu u_2 =0     &  \textnormal{on}\,\, \partial^x Q_R ,  \\ 
   u_1=0  , \,\, \partial_\nu u_2 =0    &  \textnormal{on}\,\, \partial^y Q_R ,  \\ u_2 =0  , \,\, \partial_\nu u_1 =0   &  \textnormal{on}\,\, \partial^z Q_R , \\  u=g &  \textnormal{on}\,\, \partial^s Q_R .
   \end{cases} 
\end{equation} 
Now minimize the energy $E(u, Q_R)$ over $u\in H^1(Q_R; \C )$ subject to \eqref{dir}. The standard direct method gives the existence of a minimizer $u$ defined on $Q_R$, which can be extended to a $\widetilde u \in \widetilde H^1(B_R; \C )$ reflecting according to the symmetries in Definition \ref{sobdef}. Then, it can be checked that this $\widetilde u$ has exactly the same properties stated for $u$ in Proposition \ref{existence}. In particular, $\widetilde u$ solves \eqref{GL equation} and has boundary datum $g$ on $\partial B_R$. 

Observe that from Remark~\ref{rem:symmetries} we deduce that our smooth solution $u_R$ obtained in Proposition \ref{existence} satisfies \eqref{dir}, which is just the combination of (\ref{eq: where u zero}) and (\ref{rand1}).

Hence, both approaches produce a function with the same properties. Nevertheless, being the Ginzburg-Landau energy $E(\cdot, B_R)$ not convex due to the nonconvex potential term, in both minimization problems (the one on $ \widetilde H^1_g(B_R; \C)$ and the one on $H^1(Q_R; \C )$ with mixed boundary conditions and then reflected according to the symmetries) the minimizer is generally not unique.
\end{remark}

\section{Construction of the boundary datum}\label{sec: g constr}

In this section, we need to work with differential operators on the sphere, so we first introduce some notation. We denote by $\widetilde D_R(\theta)$ a spherical disk on the sphere $\Sp^2$ with center $\theta \in \Sp^2 $ and geodesic radius $R$, and by $\nabla^\top $ the Riemannian gradient on the sphere.

Let us set
\begin{gather*}
    V:= \Sp^2 \cap X = \{\pm e_1, \pm e_2\},\\
d_V(\theta):=\dist_{\Sp^2}(\theta,V) \quad \forall \theta\in \Sp^2 \qquad \text{and}\qquad d_X(x):=\dist_{\R^3}(x,X)\quad \forall x\in\R^3.
\end{gather*}

The objective of this section is to construct a family of boundary data
\begin{equation*}
    g_R : S_R:=\partial B_R \to \R^2 , 
\end{equation*}
such that $|g_R|\leq 1$ and $(g_R)_1,(g_R)_2>0$ in $\partial^s Q_R$ (as required by Proposition~\ref{existence}), and satisfies the following properties.
\begin{itemize} 
\setlength\itemsep{5pt} 
    \item[$(i)$] $g_R$ is Lipschitz continuous on $S_R$ and satisfies the symmetries in Definition \ref{sobdef}.
    \item[$(ii)$] $g_R(x)=\min\{d_X(x),1\}g(x/|x|)$ for some $g \in \operatorname{Lip}_{\rm loc}(\Sp^2 \setminus V ; \Sp^1)$. 
    \item[$(iii)$] The degree of $g$ around $\pm e_1$ is $+1$, and the degree of $g$ around $\pm e_2$ is $-1$. 
    \item[$(iv)$] There exists a universal constant $C>0$ such that
     \begin{equation}\label{est:grad_g} 
    |\nabla^\top \ns g(\theta)|^2 \le \frac{1}{d_V(\theta)^2}+C\qquad \forall \theta\in \Sp^2.
\end{equation}
\end{itemize}

Let $\psi_1:\widetilde{D}_r(e_1)\to D_r\subset \C$ be geodesic coordinates such that
$$\psi_1(\widetilde{D}_r(e_1) \cap \{y=0,z>0\})=D_r\cap \{\Re =0 ,\Im>0\},$$
and
$$ \psi_1(\widetilde{D}_r(e_1) \cap \{z=0,y>0\})=D_r\cap \{\Im =0,\Re>0 \}.$$

We define $g$ in $\widetilde{D}_r(e_1)$ by
$$g(\theta)=\frac{\psi_1(\theta)}{|\psi_1(\theta)|}\qquad \forall \, \theta\in \widetilde{D}_r(e_1),$$
and we observe that it satisfies $g\circ \mathcal{R}_z= \overline{g}$ and $g\circ \mathcal{R}_y= -\overline{g}$ (we remark that $\widetilde{D}_r(e_1)$ is invariant under $\mathcal{R}_z$ and $\mathcal{R}_y$). Then we define $g$ also in $\widetilde{D}_r(-e_1)=-\widetilde{D}_r(e_1)$ by imposing the symmetry, namely by $g(\theta)=-\overline{g(\mathcal{R}_x \theta)}$ for every $\theta \in \widetilde{D}_r(-e_1)$. We observe that the degree of $g$ around $\pm e_1$ is equal to $+1$.

Similarly, we now fix geodesic coordinates $\psi_2:\widetilde{D}_r(e_2)\to D_r\subset \C$ such that
$$\psi_2(\widetilde{D}_r(e_2) \cap \{x=0,z>0\})=D_r\cap \{\Re =0,\Im>0\},$$
and
$$\psi_2(\widetilde{D}_r(e_2) \cap \{z=0,x<0\})=D_r\cap \{\Im =0,\Re>0 \}.$$

Then we set
$$g(\theta):=-\frac{\overline{\psi_2(\theta)}}{|\psi_2(\theta)|} \qquad \forall \, \theta\in \widetilde{D}_r(e_2),$$
and we observe that it satisfies $g\circ \mathcal{R}_z= \overline{g}$ and $g\circ \mathcal{R}_x= -\overline{g}$ (we remark that $\widetilde{D}_r(e_2)$ is invariant under $\mathcal{R}_z$ and $\mathcal{R}_x$). Again, we define $g$ also in $\widetilde{D}_r(-e_2)=-\widetilde{D}_r(e_2)$ by imposing the symmetry, namely by $g(\theta)=-\overline{g(\theta)}$ for every $\theta \in \widetilde{D}_r(-e_2)$. We observe that the degree of $g$ around $\pm e_2$ is equal to $-1$.

At this point the function $g$ is defined in four small balls around the four points in $V$, and it satisfies the symmetries in Definition~\ref{sobdef}.

Now we extend $g$ to $\Sp^2$ in such a way that it is Lipschitz continuous away from $V$ and still satisfies the symmetries. This can be done, for example, in the following way.

We consider the (closed) spherical octant $\overline{\partial^s Q_1}$, and we observe that $g$ has already been defined in $(\overline{\partial ^s Q_1} \cap \widetilde{D}_r(e_1)) \cup (\overline{\partial^s Q_1} \cap \widetilde{D}_r(e_2))$, that is the disjoint union of two quarters of disk in which $g$ is smooth away from $\{e_1,e_2\}$, and hence can be written as $g=e^{i\phi}$ for some real-valued phase $\phi$ that is smooth away from $\{e_1,e_2\}$. From the definition of $g$ we see that $\phi$ can be chosen to take values only in $[0,\pi/2]$, with $\phi=0$ where $z=0$ and $\phi=\pi/2$ where $x=0$ or $y=0$.

Now we observe that the boundary $\Gamma$ of the set $\overline{\partial ^s Q_1} \setminus (\widetilde{D}_r(e_1)\cup \widetilde{D}_r(e_2))$ in which $g$ still needs to be defined consists of the two small arcs $\gamma_1:=\Gamma\cap \partial\widetilde{D}_r(e_1)$ and $\gamma_2:=\Gamma\cap \partial\widetilde{D}_r(e_2)$, on which $g=e^{i\phi}$ has already been defined and is smooth, plus the three arcs $\gamma_x:=\Gamma\cap \{x=0\}$, $\gamma_y:=\Gamma\cap \{y=0\}$ and $\gamma_z:=\Gamma\cap \{z=0\}$.

If we now set
$$g=e^{i\phi}=\begin{cases}
i=e^{i\pi/2} &\text{on }\gamma_x\cup \gamma_y,\\
1=e^{i0} &\text{on }\gamma_z,
\end{cases}$$
then it turns out that $\phi$ is Lipschitz continuous on $\Gamma$, hence it can be extended to a Lipschitz continuous function $\phi:\overline{\partial ^s Q_1} \setminus (\widetilde{D}_r(e_1)\cup \widetilde{D}_r(e_2))\to [0,\pi/2]$, so that the map $g=e^{i\phi}$ belongs to $\operatorname{Lip}_{\rm loc}(\overline{Q_1}\setminus\{e_1,e_2\})$. Moreover, this extension can be made so that $0<\phi<\pi/2$ inside $\partial^s Q_1$, so that $g_1,g_2>0$ in $\partial^s Q_1$.

Finally, we extend $g$ to $\Sp^2\setminus V$ by reflection, imposing the symmetries of Definition~\ref{sobdef}. Since $g=i$ on $\overline{\partial ^s Q_1} \cap (\{x=0\}\cup \{y=0\})$ and $g=1$ on $\overline{\partial ^s Q_1}\cap \{z=0\}$, these reflections preserve the continuity across the planes $\{xyz=0\}$, so the resulting function is still Lipschitz continuous away from $V$, and in particular $|\nabla^\top \ns g|$ is bounded outside any neighborhood of $V$.

Moreover, inside the four disks $\widetilde{D}_r(e_1)$, $\widetilde{D}_r(e_2)$, $\widetilde{D}_r(-e_1)$, $\widetilde{D}_r(-e_2)$, the map $g$ behaves like the standard vortex map in the plane, up to a small error due to the geodesic coordinates. Therefore, since in the plane
$$\left|\nabla \left(\frac{\zeta}{|\zeta|}\right)\right|=\frac{1}{|\zeta|}\qquad \forall \, \zeta \in\C,$$
we deduce that
$$|\nabla ^\top \ns g(\theta)|^2=\frac{1}{|\psi(\theta)|^2} + O(1)=\frac{1}{d_V(\theta)^2} +O(1),$$
for $\theta\in\Sp^2$ near $V$, and this proves (\ref{est:grad_g}).

We can now set $g_R(x)=\min\{d_X(x),1\}g(x/|x|)$ and we observe that $g_R$ satisfies the symmetries in Definition~\ref{sobdef}. Moreover, since $d_X(x)\leq C R d_V(x/|x|)$ near $S_R\cap X$, then $g_R$ is Lipschitz continuous around $V$, and hence on the whole sphere $S_R$. Therefore, the maps $g_R$ and $g$ satisfy all the required properties.

\section{The entire solution}\label{sec: global sol}

Now we consider the solution $u_R$ provided by Proposition~\ref{existence} with the boundary datum $g_R:\partial B_R\to\C$ constructed in the previous section, and we wish to send $R \to +\infty$ to obtain a solution in all of $\R^3$. The next result states that this is actually works.

\begin{proposition}\label{prop: subsequence convergence} Let $u_R :B_R \to \C$ be the smooth minimizer of 
\begin{equation*}
     \min \Big\{ E(v, B_R ) \, : \,  v \in \widetilde H^1_{g_R}(B_R; \C) \Big\} 
\end{equation*}
given by Proposition \ref{existence}. Then, a discrete subsequence (not relabeled) satisfies
\begin{equation*}
    u_R \to u \,\, \mbox{ in } C^2_{\rm loc}(\R^3), 
\end{equation*}
where $u$ is a solution of \eqref{GL} in $\R^3$. Moreover
\begin{itemize}
\setlength\itemsep{3pt}
    \item[$(i)$] $|u|<1$ in $\R^3$ and $u_1, u_2 \ge 0$ in $ Q$. 
    \item[$(ii)$]  $u \not \equiv 0$ in $\R^3$ and $|u|>0$ in $\R^3 \setminus X$.
\end{itemize}
\end{proposition}

\begin{proof} By interior elliptic estimates (see for example Lemma A1 in \cite{BBH}) in $B_{R}$ we have 
\begin{equation}\label{eq534}
    \|\nabla u_R \|^2_{L^\infty(B_{R/2})} \le C \|u_R \|_{L^\infty(B_{R})}^2 \le C ,
\end{equation}
    for some universal constant $C>0$. Here we have used that $|u_R|\le 1$. Then, by Sobolev's embedding for some $p>3$ and interior Schauder estimates
    \begin{equation*}
        \| u_R\|_{C^{2,\alpha}(B_k)} \le C \| u_R\|_{C^{0,\alpha}(B_{2k})} \le C \| u_R\|_{W^{1,p}(B_{2k})} \le C , \s \forall \, k \in \mathbb{N}  \quad \forall \,  R\ge 4k, 
    \end{equation*}
    for some $C=C(k)>0$ independent of $R$. By the Arzelà-Ascoli theorem, a subsequence $\{u_{R_j} \} $ converges in $C^2(B_k)$ to a solution $u_k$ in $B_k$. Taking a further diagonal subsequence that we do not relabel, we obtain a subsequence converging in $C^2_{\rm loc}(\R^3)$ to a smooth solution $u\in C^2(\R^3)$ of \eqref{GL}. Note also that, since in \eqref{eq534} the constant $C$ is independent of $R$, letting $R \to \infty$ in \eqref{eq534} gives
    \begin{equation*}
        |\nabla u (x)| \le C , \,  \s \textnormal{for all} \,\, x\in \R^3 ,
    \end{equation*}
where $C$ is an absolute constant. 

    Since every $u_R$ given by Proposition \ref{existence} satisfies $|u_R|<1$ in $B_R$ and $(u_R)_1, (u_R)_2>0$ in $ Q_R$, it follows that $|u|\le 1$ and $u_1, u_2\ge 0$ in $ Q$. The fact that $|u|<1$ easily follows by the maximum principle. Indeed, arguing as in Remark \ref{modleone} by the strong maximum principle if $|u(x_\circ)|=1$ for some $x_\circ \in \R^3$, then $|u| \equiv  1$ everywhere. But since $u$ satisfies the symmetries in Definition \ref{sobdef} then $u=0$ on $X$, this is impossible if $|u| \equiv 1$. This proves $(i)$. 

\vsp 
Now we prove $(ii)$. Since the Ginzburg-Landau energy of $u\equiv 0$ in $B_\rho$ is $C\rho^3$, to show that $u \not \equiv 0$ it is enough to prove that $E(u, B_\rho) = o(\rho^3)$. This follows by Proposition \ref{prop: sharp energy ub} below, which is proved without using any information on $u$ apart from the fact that it is the limit (along a subsequence) in $C^2_{\rm loc}(\R^3)$ of $u_R$; hence $u \not \equiv 0$.

\vsp 
We are left to prove that $|u|>0$ in $\R^3 \setminus X$; but this follows by an argument identical to Step 4 in the proof of Proposition \ref{existence}. 

\end{proof}

\subsection{Further properties of the entire solution}
By $(ii)$ in Proposition \ref{prop: subsequence convergence}, we know that 
\begin{equation*}
    \deg\left(\frac{u}{|u|}, \sigma \right)
\end{equation*}
is well-defined for loops $\sigma$ not intersecting $X$.

\begin{lemma}\label{lem: degree of u} Let $\sigma: \Sp^1 \to \R^3 \setminus X$ be an oriented circle orthogonal to the $x$-axis or $y$-axis. Then 
\begin{equation*}
    \left| \deg\left(\frac{u}{|u|}, \sigma \right) \right| = 1 .  
\end{equation*}
\end{lemma}

\begin{proof}
By the symmetries of $u$ and $g$, it suffices to consider the case where $\sigma$ is orthogonal to the $x$-axis; using the symmetry $x\mapsto -x$, we may also assume $\sigma\subset\{x>0\}$.

Fix $R$ large so that $\sigma \Subset B_R$, and let $u_R$ be the sequence in Proposition \ref{prop: subsequence convergence} converging to the entire solution $u$. Since $u_R\to u$ in $C^2_{\rm loc}(\R^3)$, we have uniform convergence on $\sigma$; together with $|u|,|u_R|>0$ on $(\R^3\setminus X)\cap B_R$ gives that
\begin{equation*}
    \deg \left(\frac{u}{|u|},\sigma\right)=\deg \left(\frac{u_R}{|u_R|},\sigma\right),
\end{equation*}
for $R$ large. 

Let
\begin{equation*}
    \sigma_R \coloneqq \{\,t p:\ t>0,\ p\in \sigma\,\}\cap \partial B_R
\end{equation*}
be the radial (conical) projection of $\sigma$ to $\partial B_R$. The cone joining $\sigma$ to $\sigma_R$ lies in $\{x>0\}\setminus X$, so by homotopy invariance of the degree in the zero-free region of $u_R$, and by the continuity of $u_R$ up to the boundary (which follows from the continuity of $g_R$) we have
\begin{equation*}
    \deg\left(\frac{u_R}{|u_R|}, \sigma \right) = \deg\left(\frac{u_R}{|u_R|}, \sigma_R \right) = \deg\left(\frac{g}{|g|}, \sigma_R \right) = \pm 1 ,
\end{equation*}
based on the orientation of $\sigma_R$, for $R$ sufficiently large. This concludes the proof.
\end{proof}

   A similar argument to the proof of Lemma \ref{lem: degree of u} provides the degree of $u$ for every loop $\sigma \subset \R^3 \setminus X$, and results in 
   \begin{equation*}
        \left|\deg \left( \frac{u}{|u|} , \sigma \right)\right| = \operatorname{link}(\sigma,\Gamma_1)+\operatorname{link}(\sigma,\Gamma_2), 
   \end{equation*}
   where $\operatorname{link}(\cdot,\cdot)$ denotes the linking number and $\Gamma_1,\Gamma_2:\R\to\R^3$ are the lines (with a corner)
   $$\Gamma_1(t):=\begin{cases}
       (0,t,0)&\text{if }t\leq 0 ,\\
       (t,0,0)&\text{if }t\geq 0 , 
       \end{cases}
       \qquad \text{and} \qquad
       \Gamma_2(t):=\begin{cases}
       (0,-t,0)&\text{if }t\leq 0 , \\
       (-t,0,0)&\text{if }t\geq 0 . 
       \end{cases}.$$

Equivalently, one could prove that the distributional Jacobian of $u/|u|$ is (up to the usual identification of 2-forms with 1-currents via the Hodge $\star$ operator) the $1$-current induced by $\Gamma_1$ and $\Gamma_2$, with constant multiplicity equal to $\pi$.

\begin{lemma}
     The entire solution given by Theorem \ref{mainteo} is a local minimizer (i.e. in compact subsets with its own boundary datum) among competitors with the symmetries in Definition \ref{sobdef}. 
\end{lemma}
\begin{proof} Clearly it suffices to show that $u$ is locally minimizing in every ball $B_r$ for $r>0$. In what follows, $C>0$ denotes a general constant that depends only on $r$.

Recall that $u$ given by Theorem \ref{mainteo} is the limit in $C^{2}_{\rm loc}(\R^3)$ of the minimizers $u_R$ of Proposition~\ref{existence}. Since $u_R$ is a minimizer among symmetric competitors that agree with $u_R$ on $\partial B_r$, for $R \ge r$ we have that 
\begin{equation}\label{eq: u_R min in compact}
    E(u_R, B_r) \le E(w,B_r) , \quad \,\,  \forall \,  w\in \widetilde H^1(B_r; \C) \mbox{ with } w=u_R \mbox{ on } \partial B_r.  
\end{equation}

Let $v$ be any symmetric competitor of $u$ on $B_r$ with $v=u$ on $\partial B_r$. Let $U_R$ be the component-wise harmonic extension of $(u_R-u)|_{\partial B_r}$ to $B_r$. It can be checked that $U_R$ satisfies the correct symmetries since the harmonic extension is unique. In particular, since $u_R \to u$ in $C^2(B_r)$ we have that $u_R \to u $ in $H^{1/2}(\partial B_r )$, which implies 
\begin{equation*}
    \int_{B_r} |\nabla U_R|^2 =  C  [u_R-u]^2_{H^{1/2}(\partial B_r)} \to 0 .  
\end{equation*}
Moreover, by the Poincaré inequality 
\begin{equation*}
    \int_{B_r} |U_R|^2 \le 2 \int_{B_r} |U_R-(u_R-u)|^2 + 2 \int_{B_r} |u_R-u|^2 \le C \int_{B_r} |\nabla U_R|^2 +  C \|u_R-u\|^2_{H^1(B_r)} \to 0 . 
\end{equation*}
Hence $ U_R \to 0 $ in $ \widetilde H^1(B_r; \C)$ as $R \to \infty$. 

Observe that $v+U_R$ equals $u$ on $\partial B_r$ by construction. Testing \eqref{eq: u_R min in compact} with $w=v+U_R$ gives 
\begin{equation*}
    E(u_R, B_r) \le E(v+U_R, B_r) . 
\end{equation*}
By $C^2$ convergence, $E(u_R, B_r) \to E(u, B_r)$. Moreover, by $H^1$ convergence of $U_R$ to zero and continuity of the potential term, we have that $E(v+U_R, B_r) \to E(v, B_r)$. Letting $R\to \infty$ gives $E(u, B_r) \le E(v, B_r)$ as desired. 
\end{proof}

\begin{remark}
We point out that $u$ is not an unconstrained local minimizer. Indeed, if $u$ were a local minimizer, then the same would be true for the blow-downs $u_r(x):=u(rx)$, so the convergence results for local minimizers in \cite{LinRiv1,ABO} would imply that the energy densities (\ref{eq: norm en dens}) concentrate on the support of an area minimizing current.

Instead, Corollary~\ref{cor: blow down} states that the limit density is supported on $X\cap B_1$, which is not the support of an area minimizing current, since two oblique segments with the same boundary have a lower length.
\end{remark}

\section{Sharp energy asymptotics}\label{sec: sharp energy asymptotics}

In this section, we aim to prove \eqref{eq: main energy bound for u}, that is
\begin{equation*}
    \lim_{r \to \infty} \frac{E(u, B_r)}{r \log(r)} = 4\pi ,
\end{equation*}
for the solution $u:\R^3 \to \C$ constructed in Section \ref{sec: global sol}. In terms of the normalized energy densities, this implies that the limit in $B_1$ is the cross $X\cap B_1$ with density $\theta=1$, namely
\begin{equation*}
   \frac{e_{1/r}(u_r)}{\pi \log(r)} \rightharpoonup  \mathcal{H}^1 \mres  ( X\cap B_1)  \quad \mbox{in }  C^0_c(B_1)
^*.
\end{equation*}

\subsection{Upper bound}
Before proving the optimal upper bound for the energy of $u$, we need an estimate for the energy of a natural competitor for the minimization problem (\ref{minprob}) solved by $u_R$. This also provides a sharp estimate for the total energy of $u_R$ in $B_R$.

\begin{proposition}\label{prop: en bound in Br}
    Let $g :\Sp^2\to \C $ be the map constructed in Section \ref{sec: g constr}, let $m>0$ be a positive real number and let us set
    $$v_m(x):=\min\{m\cdot d_X(x),1\}g\left(\frac{x}{|x|}\right) \qquad \forall x\in \R^3.$$
    
    Then, for every $r_2>r_1\geq 0$ it holds that
    $$E(v_m, B_{r_2} \setminus B_{r_1})\leq 4\pi(r_2-r_1)\log r_2 + C(m)r_2,$$
    where $C(m)>0$ is a constant depending only on $m$.
\end{proposition}



\begin{proof}
Let us set, for brevity, $A:=B_{r_2} \setminus B_{r_1}$ and, for every $\delta>0$, let $\N_\delta:=\{d_X\leq \delta\}$.

Then 
\begin{align}
   E(v_m , A) & = E(v_m, A \setminus \N_{1/m} ) + E(v_m, A \cap \N_{1/m})  \nonumber  \\
   &=  \frac{1}{2} \int_{A \setminus \N_{1/m}} \frac{|\nabla^\top \ns g |^2 \big(\tfrac{x}{|x|} \big) }{|x|^2}  + \frac{1}{2} \int_{A \cap \N_{1/m}}  m^2|\nabla g d_X + g \nabla d_X |^2 + \int_{A \cap \N_{1/m}} \frac{(1-m^2d_X^2)^2}{4} \nonumber\\
   & \le  \frac{1}{2} \int_{A \setminus \N_{1/m}} \frac{|\nabla^\top \ns g |^2\big(\tfrac{x}{|x|} \big)}{|x|^2}  + m^2 \int_{A \cap \N_{1/m}} \frac{|\nabla^\top \ns g|^2\big(\tfrac{x}{|x|} \big) d_X^2}{|x|^2} + \left(m^2+\frac{1}{4}\right)|A \cap \N_{1/m}| \nonumber \\
   & \le  \frac{1}{2} \int_{A\setminus \N_{1/m}} \frac{|\nabla^\top \ns g |^2\big(\tfrac{x}{|x|} \big)}{|x|^2} +  m^2\int_{A \cap \N_{1/m}} \frac{|\nabla^\top \ns g|^2\big(\tfrac{x}{|x|} \big) d_X^2}{|x|^2} + Cr_2,   \label{eq: competitors eq 1}
\end{align}
where here and in the rest of the proof $C$ denotes a positive constant that may depend only on $m$, and whose precise value may vary from line to line.

We observe that 
\begin{equation}\label{eq: point bound on g and dX}
  \frac{1}{|x|^2} |\nabla^\top \ns g|^2 \ns \left( \frac{x}{|x|}\right) d_X^2(x)  =   |\nabla^\top \ns g|^2 \ns \left( \frac{x}{|x|}\right) d_X^2\left( \frac{x}{|x|}\right)  \le \left( \frac{1}{d_V^2(x/|x|)} + C \right) d_X^2\left( \frac{x}{|x|}\right)  \le C , 
\end{equation}
where we have used (\ref{est:grad_g}) and that
\begin{equation*}
    \frac{d_X(\theta)}{d_V(\theta)} \le C , \s \mbox{for every } \theta \in \Sp^2. 
\end{equation*}
Hence 
\begin{equation*}
     \int_{A \cap \N_{1/m}} \frac{|\nabla^\top \ns g|^2\big(\tfrac{x}{|x|} \big) d_X^2}{|x|^2} \le C|A \cap \N_{1/m}| \le Cr_2. 
\end{equation*}

Lastly, we have to estimate the first term in \eqref{eq: competitors eq 1}, which accounts for the highest order term in the energy. In polar coordinates, write 
\begin{equation}\label{eq: polar}
    \frac{1}{2} \int_{A \setminus \N_{1/m}} \frac{|\nabla^\top \ns g |^2\big(\tfrac{x}{|x|} \big)}{|x|^2} = \frac{1}{2} \int_{r_1}^{r_2} d\rho \int_{\partial B_1 \setminus \N_{1/(m\rho)}}  |\nabla^\top \ns g |^2(\theta) \, d \mathcal{H}^2(\theta) . 
\end{equation}
By \eqref{est:grad_g} and the coarea formula for every $\rho>2/(m\pi)$ we have
\begin{align*}
    \int_{\partial B_1 \setminus \N_{1/(m\rho)}} |\nabla^\top \ns g|^2(\theta) \, d \mathcal{H}^2(\theta) & \le \int_{ \partial B_1 \setminus \N_{1/(m\rho)} } \frac{1}{d_V(\theta)^2} \, d \mathcal{H}^2(\theta) + C \\
    &\leq \int_{1/(m\rho)}^{\pi/2} \frac{1}{t^2} \mathcal{H}^1(\{ \theta \in \Sp^2 :  d_V(\theta)=t\}) \, dt + C,
\end{align*}
while of course the left-hand side vanish if $\rho<2/(m\pi)$ (actually it is enough that $\rho<1/m$).

Since the sphere $\Sp^2$ is positively curved, for $t\in (0,\pi/2)$ it holds 
\begin{equation*}
    \mathcal{H}^1(\{ \theta \in \Sp^2 :  d_V(\theta)=t\}) \le 2\pi t \cdot 4 = 8\pi t . 
\end{equation*} 
Thus 
\begin{equation*}
    \int_{\partial B_1 \setminus \N_{1/(m\rho)}}  |\nabla^\top \ns g|^2(\theta) \, d \mathcal{H}^2(\theta) \le 8 \pi \int_{1/(m\rho)}^{\pi/2} \frac{1}{t} \, dt +  C = 8 \pi \log(m\rho) + C\leq 8 \pi \log(\rho) + C, 
\end{equation*}
and using this bound in \eqref{eq: polar} gives
\begin{equation*}
    \frac{1}{2} \int_{A \setminus \N_{1/m}} \frac{|\nabla^\top \ns g |^2\big(\tfrac{x}{|x|} \big)}{|x|^2} \le \frac{1}{2} \int_{\max\{r_1,2/(m\pi)\}}^{r_2} \big( 8 \pi \log(\rho) + C\big) \, d\rho \le 4\pi (r_2-r_1) \log(r_2) + Cr_2. 
\end{equation*}

Putting all the estimates together, we obtain the desired bound. 
\end{proof}

Since $v_1\in \widetilde{H}^{1}_{g_R}(B_R ; \C)$, the previous result immediately yields the following.

\begin{corollary}
Let $g_R \in \widetilde{H}^{1/2}(\partial B_R ; \C)$ be the boundary datum constructed in Section~\ref{sec: g constr}, and let $u_R$ be a minimizer provided by Proposition~\ref{existence}. Then
   \begin{equation}\label{minprob with gcirc}
     E(u_R, B_R)=\min \Big\{ E(v, B_R ) \, : \, v \in \widetilde H^1_{g_R}(B_R; \C) \Big\} \le 4\pi R\log(R) + O(R).
\end{equation}
\end{corollary}

As we explained in the introduction, the estimate (\ref{minprob with gcirc}) does not propagate directly to smaller scales, because of the different dependence on the radius in the monotonicity formula.

However, we can obtain the sharp upper estimate by a direct comparison argument for the minimizers $u_R$ on $B_R$ which, after letting $R\to\infty$ along the subsequence provided by Proposition \ref{prop: subsequence convergence}, gives the claimed bound for the entire solution $u$. Of course, this argument exploits in an essential way the local minimality property of $u_R$, so it could not be applied to general critical points, as opposed to the monotonicity formula, which works for every solution of (\ref{GL}).

\begin{proposition}\label{prop: sharp energy ub} Let $u$ be the entire solution of Theorem \ref{mainteo} (that is, the one given by Proposition \ref{prop: subsequence convergence}). Then 
\begin{equation*}
    \limsup_{r \to \infty} \frac{E(u, B_r)}{r \log r } \le 4\pi. 
\end{equation*}
\end{proposition}

\begin{proof}

Throughout this proof, $C>0$ denotes an absolute constant independent of $r$ and $R$ (it may vary from line to line). 

Since $u_R$ is smooth and has $(u_R)_1,(u_R)_2\geq 0$ in $Q_R$, we can write $u_R= \rho_R e^{i\varphi_R}$ in $Q_R$, for some smooth functions $\rho_R,\varphi_R:Q_R\to \R$. Similarly, we write $g=e^{i\phi}$, where $\phi:\partial^s Q_1 \to \R$ is the map constructed in Section~\ref{sec: g constr}, that we also extend to $Q_R$ by $\phi(x)=\phi(x/|x|)$. We observe that $0\leq \phi\leq \pi/2$ by construction, and also $0\leq \varphi_R \leq \pi/2$ since $u_R$ has nonnegative real and imaginary part in $Q_R$. 

Moreover, since we know that $\rho_R=0$ on $X$ and (\ref{eq534}) yields $|\nabla \rho_R| \le |\nabla u_R| \le m $ in $B_{R/2}$ for some absolute constant $m>0$ (independent of $R$), we deduce that $\rho_R \le \min\{1, m \cdot d_X\}$ in $B_{R/2}$.

Set
\[
w(x)  :=\rho_R(x) g \left(\frac{x}{|x|}\right)\qquad \forall x \in B_R,
\]
so in particular $w\in \widetilde{H}^1(B_R;\C)$ and
$$w(x)=\rho_R(x) e^{i\phi(x)} \qquad \forall x \in Q_R.$$

Let us fix $r>1$ and $\delta \in (0,1)$. Let $\eta_1:\R^3\to [0,1]$ be a smooth radial cutoff satisfying
$$\eta_1(x)=1 \quad \forall x\in B_{r+\delta r},\qquad \eta_1(x)=0 \quad \forall x\in \R^3\setminus B_{r+2\delta r},\qquad |\nabla \eta_1 (x)|\le \frac{2}{\delta r} \quad \forall x\in \R^3.$$

First, we interpolate $u_R$ and $w$. For $R>2r+4\delta r$, we consider the polar interpolation
\[
Z_1(x) := \rho_R(x)\,e^{\,i\Phi(x)}\quad\text{on } A_1:=(B_{ r+2\delta r} \setminus B_{r+\delta r})\cap Q_R,
\]
with
\[
\Phi(x):=(1-\eta_1(x)) \varphi_R(x) + \eta_1(x) \phi (x),
\]
and we extend $Z_1$ to a $B_{ r+2\delta r} \setminus B_{r+\delta r}$ by the usual symmeries.

Using the polar form and $|\nabla(\rho e^{i\theta})|^2=|\nabla\rho|^2+\rho^2|\nabla\theta|^2$ we have for the Sobolev part of the energy
\begin{align}
\frac{1}{2}\int_{A_1}|\nabla Z_1|^2 & = \frac{1}{2}\int_{A_1} |\nabla \rho_R|^2 + \frac{1}{2}\int_{A_1} \rho_R^2 |\nabla \Phi|^2 . \label{eq:grad-excess} 
\end{align}

Regarding the phase, we have
\begin{equation*}
    \nabla \Phi = (1-\eta_1) \nabla \varphi_R + \eta_1 \nabla\phi + (\phi-\varphi_R)\nabla \eta_1 , 
\end{equation*}
hence 
\begin{align*}
    |\nabla \Phi|^2 & = (1-\eta_1)^2|\nabla \varphi_R|^2 +\eta_1^2|\nabla \phi|^2  + |\nabla \eta_1|^2|\phi-\varphi_R|^2  \\ & + 2(1-\eta_1)\eta_1 \nabla \varphi_R \cdot \nabla \phi + 2 (1-\eta_1)(\phi-\varphi_R) \nabla \varphi_R \cdot \nabla \eta_1 + 2 \eta_1 (\phi -\varphi_R) \nabla \phi \cdot \nabla \eta_1 .
\end{align*}
Exploiting the estimate
$$2(1-\eta_1)\eta_1 \nabla \varphi_R \cdot \nabla \phi\leq \eta_1 ^2 |\nabla\varphi_R|^2 + (1-\eta_1)^2|\nabla \phi|^2,$$
and recalling that $|\nabla \eta_1| \le 2/(\delta r)$ and $|\phi-\varphi_R|\leq \pi$, we find that
\begin{align*}
    |\nabla \Phi|^2  & \le \big((1-\eta_1)^2+\eta_1^2 \big)|\nabla \varphi_R|^2 +\big(\eta_1^2 +(1-\eta_1)^2 \big)|\nabla \phi|^2 + \frac{C}{\delta^2 r^2} + \frac{C}{\delta r}|\nabla \varphi_R| + \frac{C}{\delta r}|\nabla \phi|  \\
    & \le |\nabla \varphi_R|^2 + |\nabla \phi|^2 + \frac{C}{\delta^2 r^2} + \frac{C}{\delta r}|\nabla \varphi_R| + \frac{C}{\delta r}|\nabla \phi| , 
\end{align*}
where we have used that $(1-t)^2+t^2 \le 1$ for $t\in [0,1]$.

Then we have 
\begin{multline*}
    \frac{1}{2}\int_{A_1}  \rho_R^2 |\nabla \Phi|^2  \le \frac{1}{2} \int_{A_1} \rho_R^2 |\nabla \varphi_R|^2 + \frac{1}{2}\int_{A_1} \rho_R^2 |\nabla \phi|^2 + \frac{Cr}{\delta} + \frac{C}{\delta r} \int_{A_1} \rho_R^2 |\nabla \varphi_R| + \frac{C}{\delta r} \int_{A_1} \rho_R^2 |\nabla \phi| \\
    \le \frac{1}{2}\int_{A_1} \rho_R^2 |\nabla \varphi_R|^2 + \frac{1}{2}\int_{A_1} \rho_R^2 |\nabla \phi|^2 + \frac{Cr}{\delta} + \frac{C\sqrt{r}}{\sqrt{\delta}} \sqrt{\int_{A_1} \rho_R^2 |\nabla \varphi_R|^2} + \frac{C\sqrt{r}}{\sqrt{\delta}} \sqrt{\int_{A_1} \rho_R^2 |\nabla \phi|^2},
\end{multline*}
where we have used that $\mathcal{H}^3(A_{1}) = Cr^3\delta $, that $\rho_R\leq 1$, and Jensen's inequality. Plugging this estimate in \eqref{eq:grad-excess} gives 
\begin{equation*}
    \frac{1}{2}\int_{A_1} |\nabla Z_1|^2 \le \frac{1}{2} \int_{A_1} |\nabla u_R|^2 + \frac{1}{2} \int_{A_1} \rho_R^2 |\nabla \phi|^2 + \frac{Cr}{\delta} + \frac{C\sqrt{r}}{\sqrt{\delta}} \sqrt{ E(u_R, A_1)} + \frac{C\sqrt{r}}{\sqrt{\delta}} \sqrt{\int_{A_1}\rho_R^2 |\nabla \phi|^2}
\end{equation*}

Now we observe that
$$\frac{1}{2}\int_{A_1}\rho_R^2 |\nabla \phi|^2\leq \frac{1}{2}\int_{A_1}\min\{1,m^2 d_X^2\} |\nabla \phi|^2 \leq\frac{1}{2} \int_{A_1} |\nabla v_m|^2=\frac{1}{8}\cdot\frac{1}{2}\int_{B_{ r+2\delta r} \setminus B_{r+\delta r}} |\nabla v_m|^2,$$
so Proposition~\ref{prop: en bound in Br} yields 
\begin{equation*}
    \frac{1}{2}\int_{A_1}\rho_R^2 |\nabla \phi|^2\leq  \frac{\pi}{2} \delta r\log (r+ 2\delta r) + C(r+2\delta r)\leq \frac{\pi}{2}\delta r\log r + Cr,
\end{equation*}
where $C$ is independent of $R$, $r$ and $\delta$. Therefore, we obtain that
\begin{align}
    \frac{1}{2}\int_{A_1} |\nabla Z_1|^2 &\le \frac{1}{2} \int_{A_1} |\nabla u_R|^2 + \frac{\pi}{2}\delta r\log r + Cr + \frac{Cr}{\delta} + \frac{C\sqrt{r}}{\sqrt{\delta}} \sqrt{ E(u_R, A_1)} + \frac{C\sqrt{r}}{\sqrt{\delta}} \sqrt{r\delta \log r + r} \nonumber\\
    &\leq \frac{1}{2} \int_{A_1} |\nabla u_R|^2 +\frac{C\sqrt{r}}{\sqrt{\delta}} \sqrt{ E(u_R, A_1)} + C \delta r\log r , \label{arsfasrf}
\end{align}
for some constant $C>0$ which is still independent of $r$ and $\delta$, provided $r>r_0=r_0(\delta)$, where $r_0$ is chosen so that
$$Cr + \frac{Cr}{\delta} + \frac{Cr}{\sqrt{\delta}} \sqrt{\delta \log r + 1} \leq \delta r\log r \qquad \forall r>r_0.$$

For the potential term, we simply have
\begin{equation}\label{eq:pot-excess}
\int_{A_1} \frac{(1-|Z_1|^2)^2}{4} = \int_{A_1} \frac{(1-\rho_R^2)^2}{4} = \int_{A_1} \frac{(1-|u_R|^2)^2}{4}
\end{equation}

Combining \eqref{arsfasrf} and \eqref{eq:pot-excess} we obtain that
\begin{equation}\label{est:Z1}
    E(Z_1, A_{1})   = \frac{1}{2}\int_{A_1} |\nabla Z_1|^2 + \int_{A_1} \frac{(1-|Z_1|^2)^2}{4} 
    \le  E(u_R, A_1) + \frac{C\sqrt{r}}{\sqrt{\delta}} \sqrt{ E(u_R, A_1)} + C\delta r \log r,
\end{equation}
for every $\delta\in (0,1)$, every $r>r_0(\delta)$ and every $R>2r+4\delta r$

Now we set
$$d_m(x):=\min \{1, m\cdot d_X(x) \},$$
and we consider the map introduced in Proposition~\ref{prop: en bound in Br}
\begin{equation*}
    v_m(x) = d_m(x) g \left(\frac{x}{|x|}\right)= d_m (x) e^{i\phi(x)},
\end{equation*}
where the second equality holds for $x\in Q_R$.

Let $\eta_2:\R^3\to [0,1]$ be a smooth radial cutoff satisfying
$$\eta_1(x)=1 \quad \forall x\in B_{r},\qquad \eta_2(x)=0 \quad \forall x\in \R^3\setminus B_{r+\delta r},\qquad |\nabla \eta_1 (x)|\le \frac{2}{\delta r} \quad \forall x\in \R^3.$$

We now consider the interpolation between $v_m$ and $w$ given by
\[
Z_2(x) := ((1-\eta_2(x))\rho_R(x)+\eta_2(x) d_m(x))e^{\,i\phi(x)}\quad\text{on }A_2:=(B_{r+\delta r}\setminus B_r)\cap Q_R,
\]
and we extend $Z_2$ to $B_{r+\delta r}\setminus B_r$ by the usual symmetries 

We have for the Sobolev part
\begin{align*}
\frac{1}{2} \int_{A_2}|\nabla Z_2|^2 & = \frac{1}{2} \int_{A_2} |\nabla ((1-\eta_2)\rho_R+\eta_2 d_m)|^2 + \frac{1}{2} \int_{A_2} ((1-\eta_2)\rho_R+\eta_2 d_m)^2 |\nabla \phi|^2 \\
& \le  \frac{1}{2}\int_{A_2} |\nabla ((1-\eta_2)\rho_R+\eta_2 d_m)|^2 + \frac{1}{2}\int_{A_2} d_m^2 |\nabla \phi|^2 ,    \label{eq:grad-excess} 
\end{align*}
since $\rho_R\le d_m$. 

By an estimate similar to the one for $\Phi$ above
\begin{equation*}
    |\nabla (1-\eta_2)\rho_R +\eta_2 d_m)|^2 \le |\nabla \rho_R |^2+|\nabla d_m|^2 +\frac{C}{\delta^2 r^2} + \frac{C}{\delta r}|\nabla \rho_R| +\frac{C}{\delta r}|\nabla d_m| , 
\end{equation*}
hence 
\begin{align*}
   \frac{1}{2} \int_{A_2}|\nabla Z_2|^2 & \le \frac{1}{2}\int_{A_2} |\nabla \rho_R|^2 + \frac{1}{2} \int_{A_2}  \bigl(|\nabla d_m|^2 +d_m^2|\nabla \phi|^2\bigl)+ \frac{Cr}{\delta} +\frac{C}{\delta r} \int_{A_2} |\nabla \rho_R| + \frac{C}{\delta r} \int_{A_2}  |\nabla d_m| \\
   & \le \frac{1}{2}\int_{A_2} |\nabla u_R|^2 + \frac{1}{2} \int_{A_2} |\nabla v_m|^2 + \frac{Cr}{\delta} + \frac{C\sqrt{r}}{\sqrt{\delta}} \sqrt{\int_{A_2} |\nabla u_R|^2} + \frac{C\sqrt{r}}{\sqrt{\delta}} \sqrt{\int_{A_2} |\nabla d_m|^2}.
\end{align*}

From the computations of Proposition~\ref{prop: en bound in Br} we find that
$$\int_{A_2} |\nabla d_m|^2 \leq C \delta r \quad \text{and}\qquad \int_{A_2} |\nabla v_m|^2\leq \frac{\pi}{2}\delta r\log r + C r,$$
so, as above, we obtain
\begin{align*}
   \frac{1}{2} \int_{A_2}|\nabla Z_2|^2 
   & \le \frac{1}{2}\int_{A_2} |\nabla u_R|^2 + \frac{Cr}{\delta} + \frac{C\sqrt{r}}{\sqrt{\delta}} \sqrt{\int_{A_2} |\nabla u_R|^2} +C\delta r \log r + Cr\\
   &\leq \frac{1}{2}\int_{A_2} |\nabla u_R|^2 + \frac{C\sqrt{r}}{\sqrt{\delta}} \sqrt{\int_{A_2} |\nabla u_R|^2} +C\delta r \log r
\end{align*}
for every $r>r_0(\delta)$, possibly for a larger $r_0$ than above (in case the constants $C$ here are larger).

For the potential part, since $d_m \ge \rho_R$ implies that $|Z_2| \ge \rho_R$, now we have
\begin{equation*}
\int_{A_2} \frac{(1-|Z_2|^2)^2}{4} \le  \int_{A_2} \frac{(1-\rho_R^2)^2}{4} = \int_{A_2} \frac{(1-|u_R|^2)^2}{4}.
\end{equation*}

Hence 
\begin{align}
    E(Z_2, A_2) & \le E(u_R, A_2) + \frac{C\sqrt{r}}{\sqrt{\delta}} \sqrt{\int_{A_2} |\nabla u_R|^2} +C\delta r \log r\nonumber\\
    & \le   E(u_R, A_2) + \frac{C\sqrt{r}}{\sqrt{\delta}} \sqrt{E(u_R,A_2)} +C\delta r \log r . \label{est:Z2}
\end{align}

Finally, by Proposition \ref{prop: en bound in Br} we find that
\begin{equation}\label{est:v}
E(v_m,B_r)\leq 4\pi r\log r + Cr\leq (4\pi +\delta)r\log r,
\end{equation}
for every $r>r_0$.

For $r>r_0$ and $R>2r+4\delta r$ define $Z\in \widetilde  H^1(B_R;\mathbb C)$ by
\[
Z=\begin{cases}
v_m & \text{in }B_r,\\
Z_2 & \text{in }B_{r+\delta r}\setminus B_{r},\\
Z_1 & \text{in }B_{r+2\delta r}\setminus B_{r+\delta r},\\
u_R & \text{in }B_R\setminus B_{r+2\delta r}.
\end{cases}
\]

Minimality of $u_R$ yields
$$E(u_R,B_R) \le E(Z,B_R) = E(v, B_r) + 8E(Z_2,A_2) + 8E(Z_1,A_1) +E(u_R,B_R\setminus B_{r+2\delta}),$$
so from (\ref{est:Z1}), (\ref{est:Z2}) and (\ref{est:v}) we deduce that for every $r>r_0$ and every $R>r+2\delta r$ it holds
\begin{align*}
E(u_R,B_{r+2\delta}) &\le (4\pi+C\delta) r\log r  + 8E(u_R, A_2)+ \frac{C\sqrt{r}}{\sqrt{\delta}} \sqrt{E(u_R,A_2)}\\
&\quad + 8E(u_R, A_1) +\frac{C\sqrt{r}}{\sqrt{\delta}} \sqrt{ E(u_R, A_1)},
\end{align*}
which implies
$$E(u_R,B_{r}) \le  (4\pi +C\delta)r \log r +\frac{C\sqrt{r}}{\sqrt{\delta}} \sqrt{ E(u_R, B_{r+2\delta r})}.$$

Let $R_j\to+\infty$ be the subsequence in Proposition \ref{prop: subsequence convergence} such that $u_{R_j}\to u$ in $C^2_{\rm loc}(\mathbb R^3)$. 
Letting $j\to+\infty$ yields
\begin{equation*}
    E(u,B_r) \le  (4\pi+C\delta) r\log r + \frac{C\sqrt{r}}{\sqrt{\delta}} \sqrt{ E(u, B_{(1+2\delta)r})}, 
\end{equation*}
for every $r>r_0(\delta)$.

On the other hand, the Lipschitz bound on $u_R$ immediately yields
$$E(u,B_r)=\lim_{j\to +\infty} E(u_{R_j},B_r)\leq \int_{B_r}\left(\frac{m^2}{2} +\frac{1}{4}\right)\leq Cr^3.$$

As a consequence, from Lemma~\ref{lem: growth reabsorbing lemma} applied with
$$f(r)=E(u,B_r),\quad A=4\pi+C\delta, \quad B= \frac{C}{\sqrt{\delta}},\quad \lambda=1+\delta ,\quad K=C,$$
we deduce that
$$\limsup_{r\to +\infty} \frac{E(u,B_r)}{r\log r}\leq 4\pi +C\delta.$$

Since $C$ is a universal constant independent of $\delta$, letting $\delta\to 0$, we conclude the proof.
\end{proof}

\begin{lemma}\label{lem: growth reabsorbing lemma}
    Let $r_0\geq 1$ be a positive number and let $f:(r_0, +\infty) \to (0, +\infty) $ be a function such that 
    \begin{equation*}
        f(r) \le A r\log r + B \sqrt{r} \sqrt{f(\lambda r)} ,  \quad \forall \, r> r_0.  
    \end{equation*}
    and
    \begin{equation*}
        f(r) \le K r^3 ,  \quad \forall \, r> r_0.  
    \end{equation*}
    for some $A,B,K>0$ and $\lambda >1$. Then there exists $r_1=r_1(A,B, \lambda )\geq r_0$ such that
    \begin{equation*}
        f(r)\le Ar\log r + r(\log r)^{2/3} + B^2\lambda^2 r \qquad \forall r>r_1.
    \end{equation*}

    In particular, it holds that
    $$\limsup_{r\to +\infty} \frac{f(r)}{r\log r}\leq A.$$
\end{lemma}

\begin{proof}

Let us fix $r_1>r_0$ sufficiently large so that
\begin{equation}\label{def:r_1}
B\sqrt{\lambda} \sqrt{A\log(\lambda r) + (\log(\lambda r))^{2/3}}\leq (\log r)^{2/3} \qquad \forall r>r_1.
\end{equation}

Now we claim that for every $n\in\mathbb{N}$ it holds that
\begin{equation}\label{claim:induction}
f(r)\leq Ar\log r + r (\log r)^{2/3} + B^{2-2^{-n}}K^{2^{-n-1}} \lambda ^{2} r^{1+2^{-n}} \qquad\forall r>r_1.
\end{equation}

We prove this by induction. First of all, we observe that our assumptions on $f$ yield
$$f(r)\leq A r\log r + B \sqrt{r} \sqrt{K\lambda^3 r^3}=A r\log r + B K^{1/2} \lambda^{3/2} r^2 \qquad \forall r > r_0,$$
and, since $\lambda>1$, this implies (\ref{claim:induction}) with $n=0$.

Now let us assume that (\ref{claim:induction}) holds for some $n\in\mathbb{N}$, so we deduce that
\begin{align*}
f(r)&\leq Ar\log r + B \sqrt{r} \sqrt{f(\lambda r)} \\
&\leq Ar\log r + B \sqrt{r} \sqrt{A\lambda r\log(\lambda r) + \lambda r (\log (\lambda r))^{2/3} + B^{2-2^{-n}}K^{2^{-n-1}} \lambda ^{2} \lambda ^{1+2^{-n}} r^{1+2^{-n}}}\\
&\leq Ar\log r + B \sqrt{r} \sqrt{A\lambda r\log(\lambda r) + \lambda r (\log (\lambda r))^{2/3}} + B\sqrt{r}\sqrt{B^{2-2^{-n}}K^{2^{-n-1}} \lambda ^{3+2^{-n}} r^{1+2^{-n}}}\\
&= Ar\log r + r B\sqrt{\lambda} \sqrt{A \log(\lambda r) +  (\log (\lambda r))^{2/3}} + B^{2-2^{-n-1}}K^{2^{-n-2}} \lambda ^{3/2+2^{-n-1}} r^{1+2^{-n-1}}
\end{align*}
for every $r>r_1$.

Now we recall that $\lambda>1$ and we observe that $3/2+2^{-n-1}\leq 2$ so, plugging (\ref{def:r_1}) into the last line, we obtain (\ref{claim:induction}) with $n+1$ in place of $n$. Hence (\ref{claim:induction}) holds for every $n\in\mathbb{N}$.

Therefore, letting $n\to +\infty$ in (\ref{claim:induction}), we conclude the proof.
\end{proof}

\subsection{Lower bound}

From here, the sharp energy bound \eqref{eq: main energy bound for u} easily follows by the clearing-out lemma (i.e., Theorem \ref{clearing out}) applied to 
\begin{equation*}
    u_r(x) = u(r x) 
\end{equation*}
and a simple covering argument. Fix $\delta<1/10$ and $\tau < 1 $. It is easy to see that there exists a family of $N$ disjoint balls $\{B_\delta(x_i)\}_{i=1}^N$ contained in $B_1$ with $x_i \in X$ and 
\begin{equation*}
   N \ge \frac{2}{\delta} - 10 .  
\end{equation*}

Since $u(x_i)=0$ for every $i$, by Theorem \ref{clearing out} we get 
\begin{equation*}
    \frac{E(u, B_r)}{r \log(r)} =  \frac{E_{1/r}(u_r, B_1)}{\log(r)} \ge \sum_i \frac{E_{1/r}(u_r, B_\delta(x_i))}{\log(r)} \ge   (2- 10  \delta)(2\pi-\kappa) \frac{\log(\delta r)}{\log(r)} , 
\end{equation*}
when $r$ is large enough, hence
\begin{equation*}
    \liminf_{r \to \infty } \frac{E(u, B_r)}{r \log(r)} \ge (2 - 10 \delta)(2\pi-\kappa) .  
\end{equation*}
Letting $\kappa \to 0^+$ and $\delta \to 0^+$ gives 
\begin{equation*}
     \liminf_{r \to \infty } \frac{E(u, B_r)}{r \log(r)} \ge 4\pi  
\end{equation*}
which, together with Proposition \ref{prop: sharp energy ub}, concludes the proof of \eqref{eq: main energy bound for u}.

\subsection{Concentration of the normalized energy densities}

We apply Theorem~\ref{thm: general convergence theorem} with $\varepsilon=1/r$ to the blow-downs $u_r$. 
Then, up to extracting a subsequence, the normalized energy measures
\begin{equation*}
    \mu_r := \frac{e_{1/r}(u_r)}{\pi\log r}
\end{equation*}
converge weakly in $C_c^0(B_1)^*$ to a limit of the form
\begin{equation*}
    \mu_\infty  =  |V|  +  f \cdot dx\mres B_1,
\end{equation*}
where $f\ge 0$ and $V$ is a stationary and rectifiable $(n-2)$-varifold with density $\theta\ge 1$, whose support contains $X\cap B_1$, so in particular $|V|\geq\theta \mathcal{H}^1\mres(X\cap B_1)$. Since $\mathcal{H}^1(X\cap B_1)=4$ (two orthogonal diameter segments in $B_1$), we have
\begin{equation*}
    |V|(B_1)\geq \theta\,\mathcal{H}^1(X\cap B_1)\ge 4,
\end{equation*}
and hence
\begin{equation*}
    4 + \int_{B_1} f \le \mu_\infty(B_1) \le \liminf_{r \to \infty} \mu_r(B_1) = 4 , 
\end{equation*}
where we have used \eqref{eq: main energy bound for u} in the last equality. Therefore $\int_{B_1} f\le 0$, which, together with $f\ge 0$, yields $f\equiv 0$. 
Moreover, $|V|(B_1)=4$ forces $\theta=1$ for $\mathcal{H}^1$-a.e.\ point of $X\cap B_1$, and $V=\mathcal{H}^1\mres(X\cap B_1)$.

We conclude that
\begin{equation*}
    \mu_\infty=\mathcal{H}^1\mres(X\cap B_1).
\end{equation*}
Since the above argument applies to any convergent subsequence, the whole sequence $\mu_r$ converges to this limit as $r\to\infty$.

\appendix

\section{Appendix}

We collect here two important results for solutions of the Ginzburg-Landau system 
\begin{equation}\tag{\(GL_\ep \)}\label{GL ep}
    -\Delta u = \frac{1}{\ep^2} u(1-|u|^2) , 
\end{equation} 
with parameter $\ep>0$.




The first one is the sharp version of the clearing-out lemma for $n=3$. This was first proved in \cite{SandierShafrir} in dimension three in the case of local minimizers, and then recently extended to every dimension \cite[Corollary 1.7]{Quantnonquant}. Previous non-sharp versions of this result were proved in the pioneering works \cite{Riv1, LinRiv1, LinRiv2, BBOsmallenergy}. In the statement $\omega_k$ denotes the Lebesgue measure of the unit ball in $\R^k$.

 \begin{theorem}[Clearing-out lemma]\label{clearing out}
 Let $\ep>0$, $x\in \R^n$ and $u_\ep : B_\delta(x) \to \C$ be a solution to \eqref{GL ep}. Then, for every $\kappa \in ( 0,\pi \omega_{n-2}) $ there exists $c(\kappa, n)>0$ such that if $\ep/\delta \le c$ and
 \begin{equation*}
     \frac{E_\ep(u_\ep, B_\delta(x))}{\delta^{n-2}} \le (\pi \omega_{n-2} - \kappa) \log(\delta/\ep) , 
 \end{equation*}
 then $|u(x)| \ge 1/2$. 
 \end{theorem}

Following the combination of the results \cite{BBH, BOSannals, DanielExistence, Quantnonquant}, we recall the asymptotic description of solutions to the Ginzburg-Landau system with parameter $\ep$ as $\ep \to 0$. 

\begin{theorem}\label{thm: general convergence theorem} Let $n\ge 3$ and $u_\ep: B_1 \to \C$ be a sequence of solutions to \eqref{GL ep} with 
\begin{equation*}
  E_\ep(u_\ep, B_1) \le C \abs{\log(\ep)} , 
\end{equation*}
for some constant $C>0$ independent of $\ep$. Then, up to subsequences, the normalized energy densities 
\begin{equation*}
    \mu_\ep := \frac{e_\ep(u_\ep)}{\pi \abs{\log(\ep)}} \, dx \mres B_1
\end{equation*}
    converge in $C_c^0(B_1)^*$ to a Radon measure $\mu$ which decomposes as
    \begin{equation*}
        \mu = |V| + f \cdot  dx \mres B_1 , 
    \end{equation*}
    where $f: B_1 \to \R$ is a smooth, nonnegative function and $V$ is a stationary, rectifiable $(n-2)$-varifold with density bounded below by $1$ on its support. Moreover, $\supp(|V|)$ is the limit of the sets $\{|u_\ep|<1/2\}$ in the local Hausdorff topology. 
\end{theorem}

\vspace{10pt}

\noindent \textbf{Acknowledgements.} The authors are grateful to Andrea Malchiodi for suggesting this problem, and to Alessandro Pigati for pointing out a mistake in a preliminary version of this work. 

The authors are members of the {\selectlanguage{italian}``Gruppo Nazionale per l'Analisi Matematica, la Probabilità e le loro Applicazioni''} (GNAMPA) of the {\selectlanguage{italian}``Istituto Nazionale di Alta Matematica''} (INdAM).

M.~C. was funded by NSF grant DMS-2304432 and is grateful to Stanford University for its kind hospitality during the realization of part of this work.

N.~P. acknowledges the INdAM-GNAMPA Project {\selectlanguage{italian} \em Gamma-convergenza di funzionali geometrici non-locali}, CUP \#E5324001950001\#.

N.~P. acknowledges the MIUR Excellence Department Project awarded to the Department of Mathematics, University of Pisa, CUP I57G22000700001.

	 	 	\bibliography{references}

\begin{thebibliography}{DdPMR22}

\bibitem[ABO05]{ABO}
G.~Alberti, S.~Baldo, and G.~Orlandi.
\newblock Variational convergence for functionals of {G}inzburg-{L}andau type.
\newblock {\em Indiana Univ. Math. J.}, 54(5):1411--1472, 2005.

\bibitem[BBH93]{BBH}
Fabrice Bethuel, Ha\"{\i}m Brezis, and Fr\'{e}d\'{e}ric H\'{e}lein.
\newblock Asymptotics for the minimization of a {G}inzburg-{L}andau functional.
\newblock {\em Calc. Var. Partial Differential Equations}, 1(2):123--148, 1993.

\bibitem[BBH94]{BBHbook}
Fabrice Bethuel, Ha\"im Brezis, and Fr\'ed\'eric H\'elein.
\newblock {\em Ginzburg-{L}andau vortices}, volume~13 of {\em Progress in Nonlinear Differential Equations and their Applications}.
\newblock Birkh\"auser Boston, Inc., Boston, MA, 1994.

\bibitem[BBO00]{BBOsmallenergy}
Fabrice Bethuel, Ha\"im Brezis, and Giandomenico Orlandi.
\newblock Small energy solutions to the {G}inzburg-{L}andau equation.
\newblock {\em C. R. Acad. Sci. Paris S\'er. I Math.}, 331(10):763--770, 2000.

\bibitem[BBO01]{BBO-asymptotics}
F.~Bethuel, H.~Brezis, and G.~Orlandi.
\newblock Asymptotics for the {G}inzburg-{L}andau equation in arbitrary dimensions.
\newblock {\em J. Funct. Anal.}, 186(2):432--520, 2001.

\bibitem[BOS06]{BOSannals}
F.~Bethuel, G.~Orlandi, and D.~Smets.
\newblock Convergence of the parabolic {G}inzburg-{L}andau equation to motion by mean curvature.
\newblock {\em Ann. of Math. (2)}, 163(1):37--163, 2006.

\bibitem[Bre23]{Brezis-op}
Ha\"im Brezis.
\newblock Some of my favorite open problems.
\newblock {\em Atti Accad. Naz. Lincei Rend. Lincei Mat. Appl.}, 34(2):307--335, 2023.

\bibitem[CEQ94]{radialsol2}
Xinfu Chen, Charles~M. Elliott, and Tang Qi.
\newblock Shooting method for vortex solutions of a complex-valued {G}inzburg-{L}andau equation.
\newblock {\em Proc. Roy. Soc. Edinburgh Sect. A}, 124(6):1075--1088, 1994.

\bibitem[CJ17]{npvortfilaments}
Andres Contreras and Robert~L. Jerrard.
\newblock Nearly parallel vortex filaments in the 3{D} {G}inzburg-{L}andau equations.
\newblock {\em Geom. Funct. Anal.}, 27(5):1161--1230, 2017.

\bibitem[CJS25]{ColJerSetrnGeodesicGL}
Andrew Colinet, Robert Jerrard, and Peter Sternberg.
\newblock Solutions of the {G}inzburg-{L}andau equations with vorticity concentrating near a nondegenerate geodesic.
\newblock {\em J. Eur. Math. Soc. (JEMS)}, 27(4):1527--1561, 2025.

\bibitem[DdPMR22]{helicalvortfilaments}
Juan D\'avila, Manuel del Pino, Maria Medina, and R\'emy Rodiac.
\newblock Interacting helical vortex filaments in the three-dimensional {G}inzburg-{L}andau equation.
\newblock {\em J. Eur. Math. Soc. (JEMS)}, 24(12):4143--4199, 2022.

\bibitem[DFP92]{DangFife92}
Ha~Dang, Paul~C. Fife, and L.~A. Peletier.
\newblock Saddle solutions of the bistable diffusion equation.
\newblock {\em Z. Angew. Math. Phys.}, 43(6):984--998, 1992.

\bibitem[DPHP24]{DPHP}
Guido De~Philippis, Aria Halavati, and Alessandro Pigati.
\newblock Decay of excess for the abelian higgs model.
\newblock {\em preprint}, 2024.

\bibitem[GT77]{GilTru}
David Gilbarg and Neil~S. Trudinger.
\newblock {\em Elliptic partial differential equations of second order}, volume Vol. 224 of {\em Grundlehren der Mathematischen Wissenschaften}.
\newblock Springer-Verlag, Berlin-New York, 1977.

\bibitem[HH94]{radialsol1}
Rose-Marie Herv\'e and Michel Herv\'e.
\newblock \'etude qualitative des solutions r\'eelles d'une \'equation diff\'erentielle li\'ee \`a{} l'\'equation de {G}inzburg-{L}andau.
\newblock {\em Ann. Inst. H. Poincar\'e{} C Anal. Non Lin\'eaire}, 11(4):427--440, 1994.

\bibitem[HT00]{HutTon}
John~E. Hutchinson and Yoshihiro Tonegawa.
\newblock Convergence of phase interfaces in the van der {W}aals-{C}ahn-{H}illiard theory.
\newblock {\em Calc. Var. Partial Differential Equations}, 10(1):49--84, 2000.

\bibitem[Ilm93]{Ilmanen-AC2Brakke}
Tom Ilmanen.
\newblock Convergence of the {A}llen-{C}ahn equation to {B}rakke's motion by mean curvature.
\newblock {\em J. Differential Geom.}, 38(2):417--461, 1993.

\bibitem[LR99]{LinRiv1}
Fanghua Lin and Tristan Rivi\`ere.
\newblock Complex {G}inzburg-{L}andau equations in high dimensions and codimension two area minimizing currents.
\newblock {\em J. Eur. Math. Soc. (JEMS)}, 1(3):237--311, 1999.

\bibitem[LR01]{LinRiv2}
Fang-Hua Lin and Tristan Rivi\`ere.
\newblock A quantization property for static {G}inzburg-{L}andau vortices.
\newblock {\em Comm. Pure Appl. Math.}, 54(2):206--228, 2001.

\bibitem[Pal79]{PalaisSymcrit}
Richard~S. Palais.
\newblock The principle of symmetric criticality.
\newblock {\em Comm. Math. Phys.}, 69(1):19--30, 1979.

\bibitem[Pis14]{Pisante14}
Adriano Pisante.
\newblock Symmetry in nonlinear {PDE}s: some open problems.
\newblock {\em J. Fixed Point Theory Appl.}, 15(2):299--320, 2014.

\bibitem[PS21]{PigStern}
Alessandro Pigati and Daniel Stern.
\newblock Minimal submanifolds from the abelian {H}iggs model.
\newblock {\em Invent. Math.}, 223(3):1027--1095, 2021.

\bibitem[PS23]{Quantnonquant}
Alessandro Pigati and Daniel Stern.
\newblock Quantization and non-quantization of energy for higher-dimensional {G}inzburg-{L}andau vortices.
\newblock {\em Ars Inven. Anal.}, pages Paper No. 3, 55, 2023.

\bibitem[Riv96]{Riv1}
Tristan Rivi\`ere.
\newblock Line vortices in the $u(1)$ - {Higgs} model.
\newblock {\em ESAIM: Control, Optimisation and Calculus of Variations}, 1:77--167, 1996.

\bibitem[Sch95]{Schatzman}
M.~Schatzman.
\newblock On the stability of the saddle solution of {A}llen-{C}ahn's equation.
\newblock {\em Proc. Roy. Soc. Edinburgh Sect. A}, 125(6):1241--1275, 1995.

\bibitem[SS17]{SandierShafrir}
Etienne Sandier and Itai Shafrir.
\newblock Small energy {G}inzburg-{L}andau minimizers in {$\mathbb{R}^3$}.
\newblock {\em J. Funct. Anal.}, 272(9):3946--3964, 2017.

\bibitem[Ste21]{DanielExistence}
Daniel Stern.
\newblock Existence and limiting behavior of min-max solutions of the {G}inzburg-{L}andau equations on compact manifolds.
\newblock {\em J. Differential Geom.}, 118(2):335--371, 2021.

\end{thebibliography}
	 	 	 \bibliographystyle{alpha}
	 	 	
\end{document}